\documentclass[12pt,a4paper]{article}

\usepackage{ifpdf}

\usepackage{latexsym}                   %Fonts
\usepackage{epsfig}                     %Figures
\usepackage{graphicx}
\usepackage{amssymb, amsmath, amsthm}   %Theorems
\usepackage{enumerate}                  %Enumeration Style
\usepackage{mathrsfs}                   %Math Fonts
\usepackage{geometry}                                                                          %fuer Layout
\usepackage{verbatim}                                                                           %Textdatei wird dargestellt (z.B. Algorithmus)

\usepackage[round]{natbib}
\usepackage{setspace}

\newtheorem{thm}{Theorem}[section]
\newtheorem{lem}[thm]{Lemma}

\newtheorem{cor}[thm]{Corollary}

\theoremstyle{definition}
\newtheorem{defi}[thm]{Definition}
\theoremstyle{definition}
\newtheorem{rem}[thm]{Remark}
\theoremstyle{definition}

\newcommand{\bE}{{\mathbb E}}

\newcommand{\bP}{{\mathbb P}}

\newcommand{\cA}{{\mathcal A}}

\newcommand{\cT}{{\mathcal T}}

\newcommand\ring[1]{\mathaccent23{#1}}
\def\Vr{\ring{V}}

\begin{document}
%%%%%%%%%%%%%%%%%%%%%%%%%%%%%%%%%%%%%%%%%%%%%%%%%%%%%%%
%\begin{frontmatter}
%\title{Speciation times in neutral models: \\ Some new analytic results}
%\author{Tanja Gernhard \thanksref{tumemail}}
%  \address{Zentrum Mathematik, Technische Universit\"at M\"unchen, 
%    Boltzmannstra\ss{}e~3, D-85747 Garching bei M\"unchen, Germany} 
%
%  
%  \thanks[tumemail]{\email{gernhard@ma.tum.de} \\ 
%The author was supported by 
%                    DFG grant ... } %TA 309/2-1. The second author was supported by DFG grant SCHA 1263/1-1.}

\thispagestyle{plain}
  \title{The conditioned reconstructed process}
  
  \author{Tanja Gernhard \\ % \footnote{The author was supported by DFG grant ...}\\
    \date{}
   {\small Department of Mathematics, Kombinatorische Geometrie (M9), TU M\"{u}nchen } \\ 
   {\small Bolzmannstr. 3, 85747 Garching, Germany} \\ 
   {\small Phone +49 89 289 16882, Fax +49 89 289 16859, gernhard@ma.tum.de}\\
  }
  \maketitle

%\maketitle

\doublespacing

\begin{abstract}
We investigate a neutral model for speciation and extinction, the constant rate birth-death process. The process is conditioned to have $n$ extant species today, we look at the tree distribution of the reconstructed trees-- i.e. the trees without the extinct species. Whereas the tree shape distribution is well-known and actually the same as under the pure birth process, no analytic results for the speciation times were known. We provide the distribution for the speciation times and calculate the expectations analytically. This characterizes the reconstructed trees completely. We will show how the results can be used to date phylogenies. %Further, we provide an easy way of sampling trees on $n$ species under the birth-death model.
% and might be useful for a prior in Bayesian analysis.
\end{abstract}


%\end{frontmatter}
\section{Introduction}
Phylogenetics is the science of reconstructing the evolutionary history of lineages (usually species). Besides providing data for systematics and for taxonomy, phylogenies are the pattern of past diversification and so can be analysed to infer past macroevolutionary process. The first common step is to compare the reconstructed trees with expectations from neutral models of diversification \citep{Gould1977,Mooers1997,Nee1992,Raup1973}.  The simplest class of neutral model are entirely homogeneous, and assume that throughout time, whenever a speciation (or extinction) event occurs, each species is equally likely to be the one undergoing that event. Of course speciation is not just random -- lineages will differ in their expected diversification rates for both instrinsic and extrinsic factors \citep{Mooers2007}. However, a neutral model is often used as a null model to analyze the data, with departures pointing the way to more sophisticated scenarios \citep{Harvey1994}.

We investigate the constant rate birth-death process \citep{Feller1968,Kendall1948b} as it is probably the most popular homogeneous model.
A birth-death process is a stochastic process which starts with an initial species. A species gives birth to a new species after exponential (rate $\lambda$) waiting times and dies after an exponential (rate $\mu$) waiting time. Throughout this paper, we will have $0 \leq \mu \leq \lambda$.
%, otherwise the process is supercritical.
 In the following, time $0$ is today and $t_{or}$ the origin of the tree, so time is increasing going into the past. Special cases of the birth-death process are the Yule model \citep{Yule1924} where $\mu = 0$ and the critical branching process \citep{AlPo2005, Popovic2004} where $\mu = \lambda$.
When looking at phylogenies, we have a given number, say $n$, of extant taxa. We therefore condition the process to have $n$ species today, we call that process the conditioned birth-death process (cBDP). The age of the tree, i.e. the time since origin of the birth-death process is $t_{or}$; if $t_{or}$ is not known, we assume a uniform prior on $(0,\infty)$ for the time of origin as it has been done in \citet{AlPo2005, Popovic2004}.
Note that a tree which evolved under a birth-death process includes extinct species, it is called the complete tree. 
From the complete tree, delete the extinct lineages. This is the reconstructed tree shape, see Figure \ref{FigLineageTree}. Label its leaves uniformly at random (since each species evolves in the same way). The resulting tree is called the reconstructed tree. Note that when reconstructing a phylogeny from (molecular) data, we see the reconstructed tree. Extinct lineages are only apparent when the fossil record is included.

\begin{figure}[!h]
\begin{center}
%\resizebox{8cm}{!}{
\includegraphics{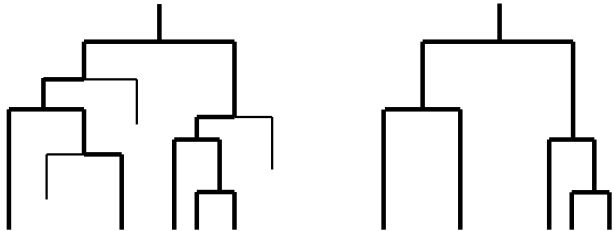}
%}
\caption{A complete tree (left) and its reconstructed tree shape (right).}
\label{FigLineageTree}
\end{center}
\end{figure}

In \citet{Nee1994}, the reconstructed tree of a birth-death process after time $t$ is discussed. However, in this paper we additionally condition on having $n$ extant species, since this allows us to compare the model with phylogenies on $n$ extant species.
We will obtain the probability for each speciation event in a tree with $n$ species.
This has been done for the Yule model and the conditioned critical branching process (cCBP) in \citet{Gernhard2007} (note that the conditioned critical branching process is the critical branching process conditioned on $n$ extant species).
For the general birth-death process, the joint probability for the shape and all speciation times has been established in \citet{Rannala1996}; the joint probability for the speciation times disregarding the shape has been established in \citet{Yang1997}. However, no individual probabilities have been established.

For establishing the individual probabilities, we introduce the point process representation for reconstructed trees (Section \ref{SecPointProc}). This had been done for the critical branching process in \citet{AlPo2005, Popovic2004}. In Section \ref{SecTimeOri}, we calculate the probability distribution of the age of a given tree on $n$ species, assuming a uniform prior on $(0,\infty)$ for the age of the tree. This enables us to derive the density function for the time of the $k$-th spceciation event in a tree with $n$ extant species (Section \ref{DensSpec}) and its expectation (Section \ref{MomSpec}) -- assuming a uniform prior or conditioning on the age of the tree.
In Section \ref{SecProperties}, we discuss some further properties of reconstructed trees. We will determine the point process when not conditioning the cBDP on the time of origin. Also, we describe the point process of the coalescent, the neutral model in population genetics. Further, we will discuss the backwards process of reconstructed trees. The backward process is the process of the coalescence of the extant species.%Finally we provide the edge length density for an edge in a given reconstructed tree.

Knowing the time of the $k$-th speciation event in a reconstructed tree with $n$ species allows us to calculate the time of a given vertex in the reconstructed tree \citep{Gernhard2006Rank,Gernhard2007}. This becomes useful for dating phylogenies. If we are able to reconstruct the phylogeny of extant species, but do not obtain speciation times, we can use the expected time of a speciation event as an estimate for the speciation time. This estimate has been used for the undated vertices in the primate phylogeny \citep{Vos2006}, assuming the Yule model. Simulations were used for obtaining the expected speciation times. We provide analytic results assuming any constant rate birth-death model. 
The methods are implemented in python as part of our PhyloTree package and can be downloaded at http://www-m9.ma.tum.de/twiki/pub/Allgemeines/TanjaGernhard/PhyloTree.zip.

In mathematical terms, a {\it reconstructed} tree is a rooted, binary tree with unique leaf labels and ultrametric edge lengths assigned, i.e. the distance from any leaf to the root is the same, see Figure \ref{FigPointProcess}, left tree. We denote the set of interior leaves by $\Vr$. A {\it ranked reconstructed tree} is a reconstructed tree without edge lengths but with a rank function defined on the interior leaves. 
A rank function is a bijection from $\Vr \rightarrow \{1,2,\ldots,|\Vr|\}$ where the ranks are increasing on any path from the root to the leaves.
A {\it(ranked) reconstructed tree shape} is a (ranked) reconstructed tree without leaf labels. 
%A {\it (ranked) tree shape} is a (ranked) reconstructed tree without the leaf labels. 
A {\it (ranked) oriented tree} is a (ranked) reconstructed tree without leaf labels but where we distinguish between the two daughter edges of the interior vertices, w.l.o.g. label them $l$ and $r$, see Figure \ref{FigPointProcess}, middle tree. Note that a (ranked) oriented tree has $n!$ possible labelings.
We introduce the oriented tree to make the proofs clearer and the statements easier. 
\begin{rem} \label{RemRecTrees}
The cBDP induces a (ranked) reconstructed tree in the following way. Consider the complete tree which evolved under the cBDP. We delete the extinct lineages and label the $n$ leaves uniformly at random with $\{1,2,\ldots,n\}$ to obtain the reconstructed tree (there are $n! 2^{-k}$ possible labelings, where $k$ is the number of cherries in the reconstructed tree shape).
%The number of possible labelings is known: Let $\tau$ be the unlabeled tree. Let $s(\tau)$ be the number of symmetrie vertices in $\tau$, i.e. the number of vertices where the two daughter trees are equal.
%There are $n! 2^{-s(\tau)}$ possible labelings for $\tau$ \citep{Steel2003}.

The interior vertices shall be ordered according to the time of speciation, this defines the rank function.
To make the reconstructed tree oriented, for each interior vertex, we label the two daughter lineages with $l$ and $r$ uniformly at random, there are $2^{n-1}$ possibities. We then ignore the leaf labelings (note that each labeling of the oriented tree is equally likely, since each labeling of the reconstructed tree was equally likely).

On the other hand, if we know the distribution on (ranked) oriented trees induced by the cBDP, we obtain the distribution on (ranked) reconstructed trees in the following way.  We choose a labeling of the leaves with $\{1,2,\ldots,n\}$ uniformly at random from the $n!$ possible labelings. We then ignore the orientation. This gives us back the distribution on (ranked) reconstructed tree.
Therefore, it is sufficient to determine the distribution on (ranked) oriented trees in order to determine the distribution on (ranked) reconstructed trees.
Overall, let $\tau_r$ be a reconstructed tree, and let $\tau_o$ be a oriented tree which was induced by $\tau_r$. Then $\bP[\tau_r] = \bP[\tau_o] 2^{n-1}/n!$, since a oriented tree has $n!$ possible labelings and for the $n-1$ interior vertices, we have the distinction between the $l$ and $r$ daughter branches.
%From \citet{Steel2003}, we know the number of rank functions for a reconstructed tree. Further, each rank function is equally likely.
%Therefore we obtain the probability of a reconstructed tree by multiplying the probability of a ranked reconstructed tree with the number of rank functions.
\end{rem}

%For comparison of the model with the data, we need to remove the orientation again. However, when calculating the speciation times in Section \ref{DensSpec} and \ref{MomSpec} for a tree, the tree may be labeled and may be oriented, but it is not necessary.
%Biologically, one could imagine that the left lineage is the ``new'' lineage, it split off from the right lineage, the ``old'' lineage.

\section{The point process} \label{SecPointProc}
In this section, we provide the density for the time of a speciation event in the reconstructed tree given $n$ species today and the time of origin being at time $t_{or}$ in the past. We do that using a point process representation.
%
%As descibed in \citet{Popovic2004,AlPo2005}, we can describe the reconstructed tree by a point process.
%Draw the tree from the top to the bottom. At each speciation event, choose one branch to be on the right and one on the left. On a horizontal axis, the leaves are located at $1, 2, \ldots n$. The speciation events are at the location $(j + 1/2,s_j)$, $j=1,2, \ldots, (n-1); s_j>0$. So a oriented tree can be described completely by the $n-1$ speciation times.
%
%Having a realization of the point process, connect the most recent speciation event with the neighboring leaves. This speciation event replaces the two neighboring leaves in the leaf set. Continue until all points are connected. This gives us back a tree.
The following point process has first been considered in connection with trees in \citet{AlPo2005,Popovic2004}.
\begin{defi}
A point process for $n$ points and of age $t_{or}$ is defined as follows. Draw the $n$ points on the horizontal axis at $1,2,\ldots,n$. Now pick $n-1$ points to be at the location $(i + 1/2,s_i)$, $i=1,2, \ldots, (n-1); 0<s_i<t_{or}$.
\end{defi}

\begin{lem} \label{LemBij}
We have a bijection between oriented trees of age $t_{or}$ and the point process of age $t_{or}$.
\end{lem}
\begin{proof}
%$''\rightarrow´´$ 
Draw the given oriented tree from the top to the bottom. At each speciation event, choose the branch with label $r$ to be on the right and with label $l$ on the left. On a horizontal axis, the leaves are located at position $1, 2, \ldots n$. The speciation events are at the location $(i + 1/2,s_i)$, $i=1,2, \ldots, (n-1); 0<s_i<t_{or}$. These are the $n-1$ points of the point process. The mapping to the point process is obviously injective and surjective, i.e. bijective.
\end{proof}
For completeness, we give the mapping from the point process to the oriented trees.
Consider a realization of the point process. Connect the most recent speciation event with the two neighboring leaves. This speciation event replaces the two neighboring leaves in the leaf set. Continue in this way until all points are connected. This gives us the corresponding oriented tree.
An example of the point process is given in Figure \ref{FigPointProcess}.

\begin{figure}[!h]
\begin{center}
%\resizebox{8cm}{!}{
\includegraphics[scale=0.8]{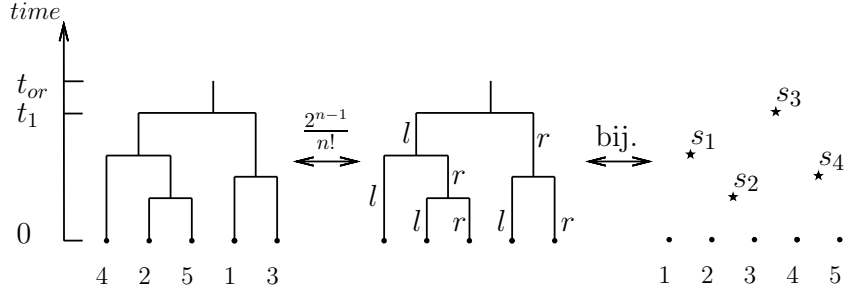}
%}
\caption{Reconstructed tree, oriented tree and the corresponding point process. The time of origin of the process is $t_{or}$, the time of the most recent common ancestor is $t_1$.}
\label{FigPointProcess}
\end{center}
\end{figure}

\begin{thm} \label{ThmTreeShape}
Each ranked oriented tree on $n$ leaves induced by the constant rate birth and death process has equal probability. 
Note that this is true with or without conditioning on the time since origin.
\end{thm}

\begin{proof}
We first determine the probability of a ranked oriented tree with leaf labels. Consider the process backwards. We have $n$ species today. We pick two species uniformly at random, which coalesce first. The one new branch gets label $l$ and the other new branch gets label $r$. Overall there are $n(n-1)$ possible choices for the first coalescent event, each one being equally likely. The two chosen leaves are replaced with their common ancestor. We proceed in this way until all species are connected. There are $n! (n-1)!$ possible scenarios for the coalescent, each one being equally likely, i.e. a ranked oriented tree with leaf labels has probability $\frac{1}{n! (n-1)!}$. Each of the $n!$ possible labelings are equally likely, therefore the probability of a ranked oriented tree is $\frac{1}{(n-1)!}$.
%Consider a ranked oriented tree $\tau$. Observe that at each speciation event we unifomly pick one of the $i$ existing lineages to speciate. Further, the edge labels $l$ and $r$ are chosen uniformly at random.
%So the probability of $\tau$ is $\prod_{i=1}^{n-1} \frac{1}{2 i} = \frac{1}{2^{n-1} (n-1)!}$.
This is the uniform distribution on ranked oriented trees of size $n$.
\end{proof}

\begin{cor} \label{CorPermutation}
Each permutation of the $n-1$ speciation points $s_1, \ldots, s_{n-1}$ in the point process of the birth-death process has equal probability.
\end{cor}
\begin{proof}
%We have a uniform distribution on ranked oriented trees (Theorem \ref{ThmTreeShape}).
We have a bijection between the oriented trees and the point process (Lemma \ref{LemBij}).
Choosing the $n-1$ speciation points $s_1, \ldots, s_{n-1}$ induces a ranked oriented tree, let the probability of that tree be $p$.
Now permute the $n-1$ speciation points arbitrary. This induces a different ranked oriented tree. Since we have a uniform distribution on ranked oriented trees (Theorem \ref{ThmTreeShape}), the probability of the new tree is again $p$.
So each permutation is equally likely.
%Since we have a bijection between the oriented trees and the point process (Lemma \ref{LemBij}), each permu.
%This means that the speciation point at position $j+1/2,s_j$ has equal probability to have rank $r=1,2,\ldots,n-1$.
%This is equivalent to each permutation of the $n-1$ speciation points $s_1, \ldots, s_{n-1}$ in the point process being equal likely.
\end{proof}
For obtaining the density of the speciation time $s_i$, we need the following results.
Under a birth-death process, the probability that a lineage leaves $n$ descendants after time $t$ is $p_n(t)$. From \citet{Kendall1949}, we know
\begin{eqnarray}
p_0(t) &=& \frac{\mu (1-e^{-(\lambda-\mu)t})}{\lambda-\mu e^{-(\lambda-\mu)t}}, \notag \\
p_1(t) &=& \frac{(\lambda-\mu)^2 e^{-(\lambda-\mu)t}}{(\lambda-\mu e^{-(\lambda-\mu)t})^2}, \notag \\
p_n(t) &=& (\lambda / \mu)^{n-1} p_1(t) [p_0(t)]^{n-1}. \label{EqnPnt}
\end{eqnarray}
Let $\tau$ be the oriented tree with $n$ leaves and $x_1 > x_2 > \ldots > x_{n-1}$ the time of the speciation events.
%, set $x=(x_1, \ldots, x_{n-1})$. 
Note that the $x_i, i =\{1,2, \ldots, n-1\}$ is the order statistic of the $s_i, i =\{1,2, \ldots, n-1\}$.

% Then the following densities are provided in \citet{Rannala1996} (Equation $6,4$),
% \begin{eqnarray*}
% f_1(n|\lambda,\mu,t_1=t) &=& (n-1) (\lambda / \mu )^{n-2} [p_0(t)]^{n-2} [p_1(t)]^2,\\
% f_1(\tau,x,n|\lambda,\mu,t_1=t) &=&  \frac{2^{n-1} \lambda^{n-2} [p_1(t)]^2 \prod_{i=2}^{n-1} p_1(x_i)}{n!}.
% \end{eqnarray*}
% The latter yields, since there are $n!(n-1)!2^{-(n-1)}$ trees on $n$ leaves and the density $f(\tau,s,n|t_1=t,\lambda,\mu)$ is independent of $\tau$,
% \begin{equation}
% f_1(x,n|\lambda,\mu,t_1=t) =  (n-1)! \lambda^{n-2} [p_1(t)]^2 \prod_{i=2}^{n-1} p_1(x_i). \notag
% \end{equation}
% From these results, we obtain
%\begin{equation}
%f_1(x|\lambda,\mu,t_1=t,n) = \frac{f(x,n|\lambda,\mu,t)}{f(n|\lambda,\mu,t)} = (n-2)! \prod_{i=2}^{n-1} \mu \frac{p_1(x_i)}{p_0(t)}. \notag
%\end{equation}
%
In \citet{Rannala1996}, joint probabilities for $x_1, \ldots, x_{n-1}$ are given.
The authors condition on the time $t_1$, the time since the most recent common ancestor ($mrca$) of the extant species.
From \citet{Yang1997}, Equation (3), we obtain the density $g$ of the ordered speciation times, $x_2 > x_3 > \ldots > x_{n-1}$, given $n$ and $x_1=t_1$,
\begin{equation}
g(x_2,x_3,\ldots,x_n|t_1=t,n) =  (n-2)! \prod_{i=2}^{n-1} \mu \frac{p_1(x_i)}{p_0(t)}. \notag
\end{equation}
The variables $x_2,x_3,\ldots,x_n$ are the order statistic of say $s_2,s_3,\ldots,s_{n-1}$.
Each permutation of the $n-2$ random variables $s_2,s_3,\ldots,s_{n-1}$ has equal probability (Corollary \ref{CorPermutation}), and therefore the density $f$ of the speciation times is,

$$f(s_2, \ldots, s_{n-1}|t_1=t,n) = \frac{g(x_2,x_3,\ldots,x_n|t_1=t,n)}{(n-2)!} = \prod_{i=2}^{n-1} \mu \frac{p_1(s_i)}{p_0(t)} $$%= \prod_{i=2}^{n-1} f(s_i|t_1=t,n)$$
which (by definition of independence) shows that the $s_i$ are i.i.d., and therefore,
\begin{equation} \label{Eqnt1}
f(s_i|t_1=t,n) = \mu \frac{p_1(s_i)}{p_0(t)} = (\lambda-\mu)^2 \frac{e^{-(\lambda-\mu)s_i}}{(\lambda-\mu e^{-(\lambda-\mu)s_i})^2} \frac{\lambda-\mu e^{-(\lambda-\mu)t}}{1 - e^{-(\lambda-\mu)t}}.
\end{equation}
Note that the expression for the density of $s_i$ does not depend on $n$, we have the same distribution for any $n$. Therefore, we do not need to condition on $n$.
%
%
%\begin{thm} \label{ThmIndeps}
%Consider a reconstructed tree with $n$ leaves, with the speciation times $s_1, \ldots, s_{n-1}$. The random variables $s_1, \ldots, s_{n-1}$ are distributed i.i.d.
%\end{thm}
%
For the distribution, we obtain by integrating Equation (\ref{Eqnt1}) w.r.t. $s_i$,
\begin{equation} \label{Eqnt1Distr}
F(s_i|t_1=t) =  \frac{1-e^{-(\lambda-\mu)s_i}}{\lambda-\mu e^{-(\lambda-\mu)s_i}} \frac{\lambda-\mu e^{-(\lambda-\mu)t}}{1 - e^{-(\lambda-\mu)t}}.
\end{equation}
Note that the probabilities are conditioned on $t_1$, the time of the $mrca$.
It is of interest to condition on $t_{or}$ instead, the time since $origin$ of the tree.
We have the-- maybe first seeming surprisingly-- property that
%\begin{thm}
\begin{equation} \label{Eqntort1}
f(s_i|t_1=t) = f(s_i|t_{or}=t).
\end{equation}
%
%Obviously, we have NO!!! We need F(s_i...) weighted with (n-2)/(n-1) cp. schmierpapier
%
%$$ f(s_i|\lambda,\mu,t_{or}=t,n) = \int_0^\infty f(s_i|\lambda,\mu,t_1,n) f(t_1|\lambda,\mu,t_{or},n) dt_1. $$
%For the cCBP and the Yule model, we verified Equation (\ref{Eqntort1}) by integration.
%
%
%However, we do not need the integration.

\begin{figure}[!h]
\begin{center}
%\resizebox{8cm}{!}{
\includegraphics[scale=0.8]{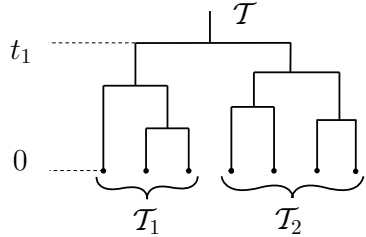}
%}
\caption{Reconstructed tree $\cT$ with daughter trees $\cT_1$, $\cT_2$. We have $mrca(\cT) = origin(\cT_1) = origin(\cT_2) = t_1$.}
\label{FigMrcaOr}
\end{center}
\end{figure}

The following argument verifies Equation (\ref{Eqntort1}).
Suppose we have a tree where the $mrca$ was at time $t_1$. The daughter trees $\cT_n,\cT_m$ of the $mrca$ have $n,m$ extant species.
The speciation times in $\cT_n,\cT_m$ occured according to Equation (\ref{Eqnt1}).
On the other hand, since the two daughter trees of the $mrca$ evolve independently, the tree $\cT_n$ can be regarded as a birth-death process which is conditioned to have $n$ species today and the time of origin was $t_{or}=t$. Therefore $f(s_i|t_1=t) = f(s_i|t_{or}=t)$, see also Figure \ref{FigMrcaOr}.
With Remark \ref{RemRecTrees}, this establishes the following theorem.

\begin{thm} \label{ThmFor} \label{ThmIndeps}
The speciation times $s_1, \ldots, s_{n-1}$ in a oriented tree (reconstructed tree) with $n$ species conditioned on the age of the tree are i.i.d. The speciation times $s_2, \ldots, s_{n-1}$ in a oriented tree (reconstructed tree) with $n$ species conditioned on the mrca are i.i.d. The time $s$ of a speciation event given (i) the time since the origin of the tree is $t_{or}$, or (ii) the time since the $mrca$ is $t_1$, has the following density and distribution,
\begin{eqnarray*}
f(s|t_{or}=t) &=&  f(s|t_1=t) = 
\left\{
\begin{array}{ll}
     \frac{ (\lambda-\mu)^2 e^{-(\lambda-\mu)s}}{(\lambda-\mu e^{-(\lambda-\mu)s})^2} \frac{\lambda-\mu e^{-(\lambda-\mu)t}}{1 - e^{-(\lambda-\mu)t}}   & \hbox{if $s \leq t$,} \\
    0 & \hbox{else,} \\
\end{array}
\right.
\\
F(s|t_{or}=t) &=& F(s|t_1=t) =
\left\{
\begin{array}{ll}
 \frac{1-e^{-(\lambda-\mu)s}}{\lambda-\mu e^{-(\lambda-\mu)s}} \frac{\lambda-\mu e^{-(\lambda-\mu)t}}{1 - e^{-(\lambda-\mu)t}},   & \hbox{if $s \leq t$,} \\
    1 & \hbox{else.} \\
\end{array}
\right.
\end{eqnarray*}
\end{thm}

Since conditioning a tree to have the $mrca$ at time $t$  can be interpreted as conditioning the two daughter trees $\cT_n$ and $\cT_m$ of $\cT$ to have the $origin$ at time $t$, we will only condition on the origin of the tree in the following.

%\begin{rem} \label{RemPointProcess}
%Theorem \ref{ThmFor} provides the density and distribution for a speciation event. As descibed in \citet{Popovic2004,AlPo2005}, we can describe the reconstructed tree by a point process.
%Draw the tree from the top to the bottom. At each speciation event, choose one branch to be on the right and one on the left. On a horizontal axis, the leaves are located at $1, 2, \ldots n$. The speciation events are at the location $(j + 1/2,s_j)$, $j=1,2, \ldots, (n-1); s_j>0$. The speciation events are the points which are distributed on the vertical axis according to Theorem \ref{ThmFor}, in particular the $n-1$ points are i.i.d.
%\end{rem}

\subsection{Special models}
\subsubsection{The Yule model}
For the special case of a pure birth process, i.e. $\mu = 0$, which is the Yule model,
Equation (\ref{Eqnt1}) simplifies to
\begin{eqnarray*} 
f(s|t) &=&  \frac{\lambda e^{- \lambda s}}{1- e^{-\lambda t}}  \\
F(s|t) &=&  \frac{1- e^{- \lambda s}}{1- e^{-\lambda t}} 
\end{eqnarray*}
which has already been established in \citet{Nee2001}-- he conditioned on the time since the $mrca$ though.

\subsubsection{The conditioned critical branching process}
In a cCBP, we have $\lambda=\mu$. As $\mu \rightarrow \lambda$, we get in the limit using Equation (\ref{Eqnt1}), (\ref{Eqnt1Distr}) and (\ref{Eqntort1}), and the property $e^{-\epsilon} \sim 1-\epsilon$ for $\epsilon \rightarrow 0$,
\begin{eqnarray*}
f(s|t) &=&   \frac{1}{(1+ \lambda s)^2} \frac{1+\lambda t}{t}, \\
F(s|t) &=&   \frac{s}{1+ \lambda s} \frac{1+\lambda t}{t}.
\end{eqnarray*}
This has already been established in a different way for $\lambda = 1$ in \citet{AlPo2005,Popovic2004}.

\section{The time of origin} \label{SecTimeOri}

Suppose nothing is known about $t$, the time of origin of a tree. As in \citet{AlPo2005,Popovic2004}, we then assume a uniform prior on $(0,\infty)$, i.e. a tree is equally likely to origin at any point in time. Note that the prior does not integrate to $1$. For any constant function, the integral is $\infty$. Therefore the prior is not a density. Such a prior is called improper; a discussion and justification is found e.g. in \citet{Berger1980}. Assuming the uniform prior, we will establish the density for $t$ given $n$ extant species.
% From \citet{Nee1994}, we obtain the probability of $n$ species at time $t$ after origin, $\bP_{or}[n|t]$,
% \begin{eqnarray*}
% \bP_{or}[0|t] &=& 1-P(t), \\
% \bP_{or}[n|t] &=& P(t)(1-u(t))u(t)^{n-1} = \lambda^{n-1} (\lambda-\mu)^2 \frac{(1-e^{-(\lambda-\mu)t)^{n-1}} e^{-(\lambda-\mu)t}}{(\lambda-\mu e^{-(\lambda-\mu)t})^{n+1}}.
% \end{eqnarray*}
% with the functions $P(t), u(t)$ defined as
% \begin{eqnarray*}
% P(t) &=& \frac{\lambda-\mu}{\lambda-\mu e^{-(\lambda-\mu)t}}, \qquad u(t) = \lambda \frac{1-e^{-(\lambda-\mu)t}}{{\lambda-\mu e^{-(\lambda-\mu)t}}}.% \label{Eqnut}
% \end{eqnarray*
%}
From Equation (\ref{EqnPnt}), we have the probability of $n$ extant species given the time of origin is $t$,
$$\bP_{or}[n|t] = \lambda^{n-1} (\lambda-\mu)^2 \frac{(1-e^{-(\lambda-\mu)t)^{n-1}} e^{-(\lambda-\mu)t}}{(\lambda-\mu e^{-(\lambda-\mu)t})^{n+1}}.$$
In order to derive the density for $t$ given $n$, we need the following lemma.
\begin{lem} \label{LemIntPntor}
Let $\bP_{or}[n|t]$ be the probability that a tree has $n$ extant species given the time of origin $t$. We have
$$\int_0^\infty \bP_{or}[n|t] dt = \frac{1}{n \lambda}.$$
\end{lem}

\begin{proof}
The derivative of $\left( \frac{1-e^{-(\lambda-\mu)t}}{\lambda-\mu e^{-(\lambda-\mu)t}} \right)^{n}$ is, using the quotient rule,
\begin{eqnarray*}
\frac{d}{dt} \left( \frac{1-e^{-(\lambda-\mu)t}}{\lambda-\mu e^{-(\lambda-\mu)t}} \right)^{n} &=& n \frac{(1-e^{-(\lambda-\mu)t)^{n-1}}}{(\lambda-\mu e^{-(\lambda-\mu)t})^{n+1}} (\lambda-\mu)^2 e^{-(\lambda-\mu)t}
\end{eqnarray*}
and therefore,
\begin{eqnarray*}
\int_0^\infty \bP_{or}[n|t] dt &=&  \frac{\lambda^{n-1}}{n}     \left[ \left( \frac{1-e^{-(\lambda-\mu)t}}{\lambda-\mu e^{-(\lambda-\mu)t}} \right)^{n} \right]_0^\infty \\
&=&  \frac{\lambda^{n-1}}{n}  \left( \frac{1}{\lambda^n} - 0 \right) =  \frac{1}{\lambda n} 
\end{eqnarray*}
which establishes the lemma.
% For solving the integral
% $$ \int_0^\infty \bP[n|t] dt = \int_0^\infty  \lambda^{n-1} (\lambda-\mu)^2 \frac{(1-e^{-(\lambda-\mu)t)^{n-1}} e^{-(\lambda-\mu)t}}{(\lambda-\mu e^{-(\lambda-\mu)t})^{n+1}} dt, $$
% we substitude $x := \lambda - \mu e^{-(\lambda-\mu)t}$, i.e. $dx = \mu (\lambda-\mu) e^{-(\lambda-\mu)t}$.
% Therefore,
% \begin{eqnarray*}
% \int_0^\infty \bP[n|t] dt &=&  \lambda^{n-1} \frac{\lambda-\mu}{\mu}  \int_{\lambda-\mu}^{\lambda} \frac{\left(\frac{x-(\lambda-\mu)}{\mu}   \right)^{n-1}}{x^{n+1}} dx \\
% &=&  \lambda^{n-1} \frac{\lambda-\mu}{\mu}   \int_{\lambda-\mu}^{\lambda} \frac{\sum_{i=0}^{n-1} {n-1 \choose i} x^i \frac{(-(\lambda-\mu))^{n-i-1}}{\mu^{n-1}}  }  {x^{n+1}} dx \\
% &=&  - \frac{ \lambda^{n-1}}{\mu^n}  \sum_{i=0}^{n-1} {n-1 \choose i} (-(\lambda-\mu))^{n-i} \left[-\frac{1}{(n-i)x^{n-i}} \right]_{\lambda-\mu}^{\lambda}  \\
% &=&  \frac{ \lambda^{n-1}}{n \mu^n}  \sum_{i=0}^{n-1} {n \choose i} (-(\lambda-\mu))^{n-i} ( (\lambda-\mu)^{i-n} - \lambda^{i-n} ) \\
% &=&  \frac{ \lambda^{n-1}}{n \mu^n}  \sum_{i=0}^{n-1} {n \choose i} (-1)^{n-i}  - 
% \frac{1}{n \lambda \mu^n}  \sum_{i=0}^{n-1} {n \choose i} (-(\lambda-\mu))^{n-i} \lambda^{i}  \\
% &=&  \frac{ \lambda^{n-1}}{n \mu^n} ( (1-1)^n  - 1) -
% \frac{ 1}{n \lambda \mu^n}  (\mu^n - \lambda^n)  \\
% &=& \frac{1}{n \lambda}
% \end{eqnarray*}

\end{proof}

\begin{thm} \label{ThmPtorn}
%Let $t_{or}$ be the time of origin of a tree.
%Let $q_{or}(t)$ be the density function of $t_{or}$. 
We assume the uniform prior on $(0, \infty)$ for the time of origin of a tree. Conditioning the tree on having $n$ species today, the time of origin has density function
\begin{equation}
q_{or}(t|n) =  n \lambda^{n} (\lambda-\mu)^2 \frac{(1-e^{-(\lambda-\mu)t})^{n-1}  e^{-(\lambda-\mu)t}}{(\lambda-\mu e^{-(\lambda-\mu)t})^{n+1}}. \label{EqnPtorn}
\end{equation}
\end{thm}
\begin{proof}
With Bayes' law, we have
\begin{eqnarray}
q_{or}(t|n) &=& \frac{\bP_{or}[n|t] q_{or}(t)}{\bP_{or}[n]} =\frac{\bP_{or}[n|t] q_{or}(t)}{\int_0^\infty \bP_{or}[n,t] dt} \notag \\
&=& \frac{\bP_{or}[n|t] q_{or}(t)}{\int_0^\infty \bP_{or}[n|t] q_{or}(t) dt } =  \frac{\bP_{or}[n|t]}{\int_0^\infty \bP_{or}[n|t]dt} \notag \\
&\stackrel{(\ref{LemIntPntor})}{=}& \lambda n \bP_{or}[n|t] \notag \\
&=& n \lambda^{n} (\lambda-\mu)^2 \frac{(1-e^{-(\lambda-\mu)t})^{n-1}  e^{-(\lambda-\mu)t}}{(\lambda-\mu e^{-(\lambda-\mu)t})^{n+1}}. \notag
\end{eqnarray}
\end{proof}

\begin{cor} \label{CorFtorn}
The distibution for the time of origin given $n$ species today is
$$ Q_{or}(t|n) = \left(  \frac{\lambda(1-e^{-(\lambda-\mu)t})}{\lambda-\mu e^{-(\lambda-\mu)t}} \right)^n.$$
\end{cor}
\begin{proof}
%We will show that the given $Q_{or}(t|n)$ is the antiderivative of $q_{or}(t|n)$ in Theorem \ref{ThmPtorn}. 
We have  $\lim_{t \rightarrow \infty} Q_{or}(t|n) = 1$.
Differentiation of $Q_{or}(t|n)$ w.r.t. $t$ yields $\frac{d}{dt} Q_{or}(t|n) =  n \lambda^{n} (\lambda-\mu)^2 \frac{(1-e^{-(\lambda-\mu)t})^{n-1}  e^{-(\lambda-\mu)t}}{(\lambda-\mu e^{-(\lambda-\mu)t})^{n+1}} = q_{or}(t|n)$
which completes the proof.
\end{proof}

\section{The time of speciation events} \label{DensSpec}
In this section, we calculate the density for the time of the $k$-th speciation event given we have $n$ species today.
Knowing that the distribution on ranked reconsructed trees is uniform (Theorem \ref{ThmTreeShape}), this characterizes the reconstructed trees completely. These results allow us to calculate the density for the time of a given vertex in a reconstructed tree \citep{Gernhard2006Rank,Gernhard2007}.
\subsection{Known age of the tree} \label{DensSpect}
Let $\cA_{n,t}^k$ be the time of the $k$-th speciation event in a reconstructed tree $\cA$ with $n$ extant species and age $t$.
The $n-1$ speciation events in $\cA$ are i.i.d. and have the density function $f(s|t)$, see Theorem \ref{ThmFor}.
The density of $\cA_{n,t}^k$ is therefore the $(n-k)$-th order statistic, which is (see e.g. \citet{DehlingStochastik}, Theorem 9.17),
\begin{eqnarray} 
f_{\cA_{n,t}^k} (s) &=& (n-k) {n-1 \choose n-k} F(s|t)^{n-k-1} % \notag \\
(1-F(s|t))^{k-1} f(s|t) \label{EqnfAntk} %\notag \\
%&=& k {n-1 \choose k}  (\lambda-\mu)^{k+1} (e^{-(\lambda-\mu)s} - e^{-(\lambda-\mu)t})^{k-1} e^{-(\lambda-\mu)s}   \ldots \notag \\
%& & \frac{ (\lambda-\mu e^{-(\lambda-\mu)t})^{n-k} (1-e^{-(\lambda-\mu)s})^{n-k-1} }{(\lambda-\mu e^{-(\lambda-\mu)s})^{n} (1- e^{-(\lambda-\mu)t})^{n-1}} 
\end{eqnarray}
for $s \leq t$ and $f_{\cA_{n,t}^k} (s) = 0$ else.
The distribution function of $\cA_{n,t}^k$ is
\begin{equation}
F_{\cA_{n,t}^k} (s) = \sum_{i=0}^{k-1} {n-1 \choose i} F(s|t)^{n-i-1} (1-F(s|t))^{i} \label{EqnFAntk}
\end{equation}
for $s \leq t$ and $F_{\cA_{n,t}^k} (s) = 1$ else.

\subsection{Unknown age of the tree}
If the time of origin is unknown, we assume a uniform prior for the time of origin. 
%In Theorem \ref{ThmPtorn}, we obtained the density function for $t$ given $n$, $q_{or}(t|n)$.
Using this assumption, we will calculate the density function for $\cA_n^k$, the time of the $k$-th speciation event in a tree with $n$ extant species.

\begin{thm} \label{ThmfAnk}
Let $\cA_n^k$ be the time of the $k$-th speciation event in a tree with $n$ extant species. We have for $0 \leq \mu < \lambda$,
$$f_{\cA_n^k}(s) =  (k+1) {n \choose k+1 }  \lambda^{n-k} (\lambda-\mu)^{k+2} e^{-(\lambda-\mu)(k+1)s}  \frac{(1- e^{-(\lambda-\mu)s})^{n-k-1}}{(\lambda-\mu  e^{-(\lambda-\mu)s})^{n+1}}. $$
% $$f_{\cA_n^k}(s) =  (k+1) {n \choose k+1 } C  q^{n-k} (1-q)^{k} $$ 
%$$f_{\cA_n^k}(s) =  \lambda \frac{\lambda-\mu}{\lambda-\mu  e^{-(\lambda-\mu)s}} (k+1) {n \choose k+1 }   q^{n-(k+1)} (1-q)^{k+1} $$ 
%where $q = Q_{or}(s|1) = \frac{\lambda(1- e^{-(\lambda-\mu)s}) }{\lambda-\mu e^{-(\lambda-\mu)s} }$. % and $C=(\lambda-\mu)^2 \frac{e^{-(\lambda-\mu)s} }{( \lambda-\mu e^{-(\lambda-\mu)s})(1-e^{-(\lambda-\mu)s }) }$.
\end{thm}

\begin{proof}
For a fixed time $t$ of origin, we have the density $f_{\cA_{n,t}^k}$ for the time of the $k$-th speciation event (Section \ref{DensSpect}). With our uniform prior, the time of origin has density function $q_{or}(t|n)$. The density $f_{\cA_n^k}(s)$ is therefore
%\begin{equation} f_{\cA_n^k}(s)  =  \int_0^\infty f_{\cA_{n,t}^k}(s) q_{or}(t|n) dt. \label{EqnDens1} \end{equation}
%We may change the order of differentiation and integration (Lemma \ref{LemSwitch}). Therefore,
\begin{eqnarray*}
f_{\cA_n^k}(s) &=& \int_s^\infty f_{\cA_{n,t}^k}(s) q_{or}(t|n) dt \\
%&=& \int_s^\infty  f_{\cA_{n,t}^k}(s) q_{or}(t|n) dt \\
&=&  \int_s^\infty k {n-1 \choose k}  (\lambda-\mu)^{k+1} (e^{-(\lambda-\mu)s} - e^{-(\lambda-\mu)t})^{k-1} e^{-(\lambda-\mu)s}   \times \notag \\
& & \frac{ (\lambda-\mu e^{-(\lambda-\mu)t})^{n-k} (1-e^{-(\lambda-\mu)s})^{n-k-1} }{(\lambda-\mu e^{-(\lambda-\mu)s})^{n} (1- e^{-(\lambda-\mu)t})^{n-1}}  \times \notag \\
& & n \lambda^{n} (\lambda-\mu)^2 \frac{(1-e^{-(\lambda-\mu)t})^{n-1}  e^{-(\lambda-\mu)t}}{(\lambda-\mu e^{-(\lambda-\mu)t})^{n+1}}dt \notag \\
%&=& \int_s^\infty  n k {n-1 \choose k}  (\lambda-\mu)^{k+3} \lambda^{n}  \frac{ e^{-(\lambda-\mu)ks} (1-e^{-(\lambda-\mu)s})^{n-k-1} }{(\lambda-\mu e^{-(\lambda-\mu)s})^{n} }   \times \notag \\
%& & \frac{(1 - e^{-(\lambda-\mu)(t-s)})^{k-1} e^{-(\lambda-\mu)t}}{ (\lambda-\mu e^{-(\lambda-\mu)t})^{k+1} }\\
&=&  n k {n-1 \choose k}  (\lambda-\mu)^{k+3} \lambda^{n}  \frac{ e^{-(\lambda-\mu)ks} (1-e^{-(\lambda-\mu)s})^{n-k-1} }{(\lambda-\mu e^{-(\lambda-\mu)s})^{n} }  \times \notag \\
& & \int_s^\infty \frac{(1 - e^{-(\lambda-\mu)(t-s)})^{k-1} e^{-(\lambda-\mu)t}}{ (\lambda-\mu e^{-(\lambda-\mu)t})^{k+1} } dt \\
&=& n k {n-1 \choose k}  (\lambda-\mu)^{k+3} \lambda^{n}  \frac{ e^{-(\lambda-\mu)ks} (1-e^{-(\lambda-\mu)s})^{n-k-1} }{(\lambda-\mu e^{-(\lambda-\mu)s})^{n} }  \times \notag \\
& & \frac{e^{-(\lambda-\mu)s}}{k (\lambda - \mu) (\lambda - \mu e^{-(\lambda-\mu)s})} \left[ \left( \frac{1 - e^{-(\lambda-\mu)(t-s)}}{\lambda-\mu e^{-(\lambda-\mu)t}} \right)^{k}  \right]_s^\infty \\
%&=& n  {n-1 \choose k}  (\lambda-\mu)^{k+2} \lambda^{n}  \frac{ e^{-(\lambda-\mu)(k+1)s} (1-e^{-(\lambda-\mu)s})^{n-k-1} }   {(\lambda-\mu e^{-(\lambda-\mu)s})^{n+1} }  \left[ \frac{1}{\lambda^k} - 0 \right] \notag \\
&=& (k+1) {n \choose k+1 }  \lambda^{n-k} (\lambda-\mu)^{k+2} e^{-(\lambda-\mu)(k+1)s}  \frac{(1- e^{-(\lambda-\mu)s})^{n-k-1}}{(\lambda-\mu  e^{-(\lambda-\mu)s})^{n+1}}
\end{eqnarray*}
which establishs the theorem.
\end{proof}

\begin{rem} \label{RemDensSpecModels}
Under the Yule model, i.e. setting $\mu=0$ and $\lambda$ arbitrary in Theorem \ref{ThmfAnk}, we have,
$$f_{\cA_n^k}(s) =  (k+1) {n \choose k+1 } \lambda \frac{(e^{\lambda s} -1)^{n-k-1}}{e^{\lambda s n}}$$
which has been established in \citet{Gernhard2007} for $\lambda=1$ in a different way.

In the cCBP,
% \citep{AlPo2005},
 the birth rate equals the death rate, $\lambda = \mu$. Taking the limit $\mu \rightarrow \lambda$ in $f_{\cA_n^k}$,we obtain from Theorem \ref{ThmfAnk} using the property $e^{-\epsilon} \sim 1-\epsilon$,

$$f_{\cA_n^k}(s) =  (k+1) {n \choose k+1 } \lambda^{n-k} \frac{s^{n-k-1}}{(1+ \lambda s)^{n+1}}$$
which has been established in \citet{Gernhard2007} for $\lambda=1$ in a direct way.
\end{rem}

\section{Expected speciation times} \label{MomSpec}
In this section, we calculate the expected time of the $k$-th speciation event in a reconstructed tree with $n$ species analytically. Our Python implementation for dating trees uses the analytic results.
Higher moments are calculated numerically.
\subsection{Known age of the tree}
\begin{thm} \label{ThmExpt}
The expectation of $\cA_{n,t}^k$ is, for $0 < \mu < \lambda$,
\begin{eqnarray*}
\bE[\cA_{n,t}^k]&=& t -  \sum_{i=0}^{k-1} \sum_{j=0}^i {n-1 \choose i} {i \choose j} (-1)^{i+j} \left( \frac{\lambda-\mu e^{-(\lambda-\mu)t}}{1 - e^{-(\lambda-\mu)t}} \right)^{n-j-1} \times \\
& & \left[ g(j)+ \sum_{l=1}^{n-j-1} \sum_{m=0}^{l-1}  {n-j-1 \choose l}  {l-1 \choose m} (-1)^{l+m} \frac{\lambda^{l-1-m}}{(\lambda-\mu) \mu^l} h(j,m) \right]
\end{eqnarray*}
where 
\begin{eqnarray*}
g(j)&=& \frac{1}{(\lambda-\mu)\lambda^{n-j-1}} \times \\
& &\left[ \ln \left( \frac{\lambda e^{(\lambda-\mu)t} - \mu}{\lambda-\mu}  \right) - \sum_{m=1}^{n-j-2} {n-j-2 \choose m} \frac{\mu^m}{m} \left( \lambda e^{(\lambda-\mu)t} - \mu)^{-m} - (\lambda-\mu)^{-m}  \right) \right]
\end{eqnarray*}
and
\begin{eqnarray*}
h(j,m) &=&  
\left\{
\begin{array}{ll}
     \ln \frac{ \lambda - \mu e^{-(\lambda-\mu)t} }{\lambda-\mu},   & \hbox{if $m+j+1-n = -1$,} \\
     \frac{ (\lambda - \mu e^{-(\lambda-\mu)t})^{m+j+2-n} -  (\lambda - \mu)^{m+j+2-n} }{m+j+2-n}  & \hbox{else.} \\
\end{array}
\right.
\end{eqnarray*}
For $\mu = 0$, we have,
\begin{eqnarray*}
\bE[\cA_{n,t}^k] &=&   \sum_{i=0}^{n-k-1} \sum_{j=0}^{k-1} \frac {k {n-1 \choose k}{n-k-1 \choose i}{k-1 \choose j}(-1)^{i+j}}{ \lambda (k+i-j)^{2}} \times \\
& & (1-e^{-\lambda t})^{1-n} (e^{-j \lambda t} -  ( (k+i-j)\lambda t + 1)e^{-(k+i)\lambda t}).
\end{eqnarray*}
For $\mu = \lambda$, we have
\begin{eqnarray*}
\bE[\cA_{n,t}^k] &=&  t - \sum_{i=0}^{k-1} \sum_{j=0}^i {n-1 \choose i} {i \choose j} \frac{(-1)^{i+j}}{\lambda^{n-j}} \left( \frac{1+\lambda t}{t}\right)^{n-j-1} \times \\
& & \left[  \lambda t  - (n-j-1) \ln(1+\lambda t) + \sum_{l=2}^{n-j-1} {n-j-1 \choose l} (-1)^l \frac{(1+\lambda t)^{-l+1}-1}{1-l} \right].
\end{eqnarray*}
\end{thm}
The proof is found in the appendix.

\subsection{Unknown age of the tree}
A closed form solution for the first and second moment (for all $k$) of $\cA_n^k$ under the Yule and for all moments under the cCBP (with the setting $\lambda=1$) is given in \citet{Gernhard2007},
\begin{eqnarray}
\bE^{Yule}[\cA_n^k] &=& \sum_{i=k+1}^n \frac{1}{i}, \label{EqnExpYule} \\
\bE^{Yule}[(\cA_n^k)^2] &=& \sum_{i=k+1}^n \frac{1}{i^2} + \sum_{i=k+1}^n  \sum_{j=k+1}^n \frac{1}{ij},\\
\bE^{cCBP}[(\cA_{n}^k)^m] &=&
\left\{
\begin{array}{ll}
    \frac{{n-k+m-1 \choose m}}{{k \choose m}} & \hbox{if $k \geq m$},  \label{EqnExpCBP}\\
    \infty & \hbox{else}. \\
\end{array}
\right.
\end{eqnarray}
For general $\lambda,\mu$, we have the following analytic expression for the expectation (the proof is found in the appendix).
\begin{thm} \label{ThmExp}
For $0 < \mu < \lambda$, the moments of $\cA_n^k$ are ($\rho := \mu / \lambda$),
\begin{eqnarray*}
\bE[\cA_n^k] &=&  \frac{k+1}{\lambda} {n \choose k+1 } (-1)^{k} \sum_{i=0}^{n-k-1} {n-k-1 \choose i} \frac{1}{(k+i+1) \rho} \left(\frac{1}{\rho} -1 \right)^{k+i} \times \\
& & \left[ \log \left( \frac{1}{1-\rho}  \right)  -  \sum_{j=1}^{k+i} {k+i \choose j} \frac{(-1)^{j}}{j}  \left(1 - \left(\frac{1}{1-\rho}\right)^{j} \right)  \right].
\end{eqnarray*}
For $\mu=0$ we have
 $$\bE[\cA_n^k] = \sum_{i=k+1}^n \frac{1}{\lambda i}$$
and for $\mu = \lambda$ we have
 $$\bE[\cA_n^k] = \frac{n-k}{\lambda k}.$$
In particular, the expectations basically only depend on $\rho$. Different $\lambda$ just scale time by $1 / \lambda$.
\end{thm}

\begin{figure}[!h]
\begin{center}
%\resizebox{8cm}{!}{
\includegraphics[scale=0.8]{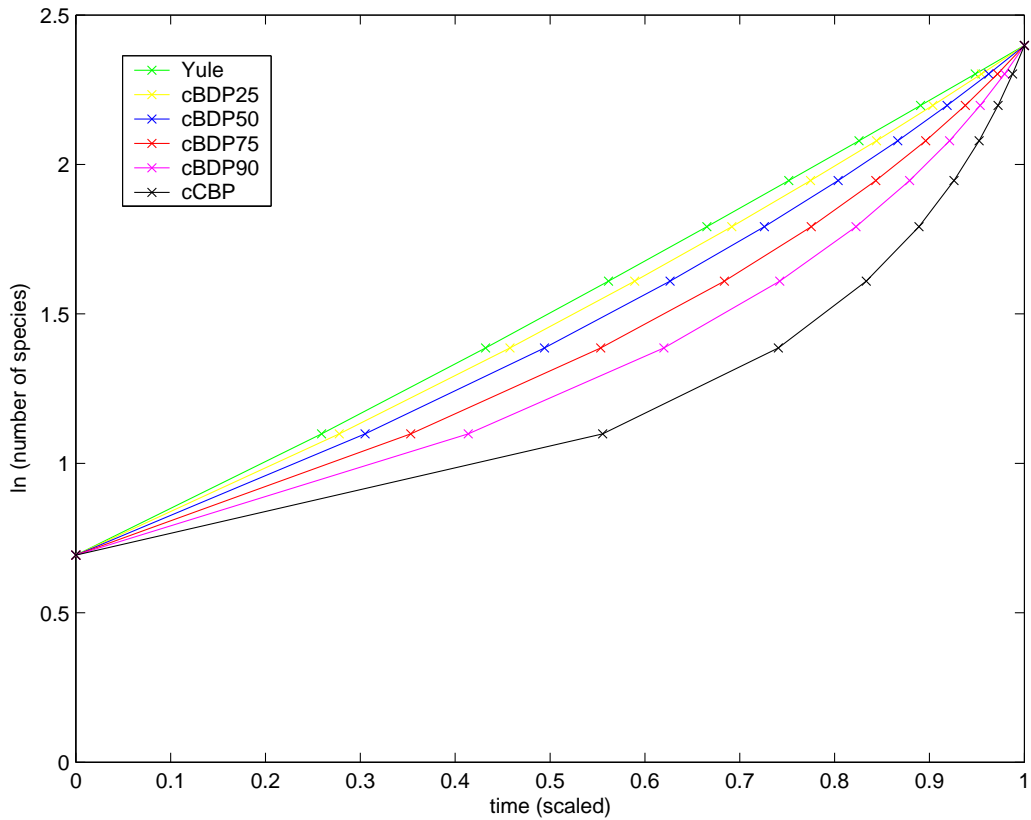}
%}
\caption{Lineage through time plot for $n=10$ species. Time is scaled such that the $mrca$ is at $0$ and today is $1$. We have $\rho = 0, 0.25, 0.5,0.75,0.9,1$ from top to bottom. Note that for variing $\lambda$, time is scaled by $1/\lambda$ (compared to $\lambda=1$). That means, since we scale time, the plots are the same for any $\lambda$.}
\label{FigLTTexpectedh}
\end{center}
\end{figure}
Knowing the expected time of the $k$-th speciation event allows us to draw a lineage-through-time (LTT) plot \citep{Nee1994} analytically. In a LTT plot, the time vs. the number of species at that time (on a logarithmic scale) is drawn. These plots are frequently used for comparing the data with a model. Commonly, the LTT plots for different models are obtained via simulations.
Since we know the expected time of the speciation events in our model analytically, we can plot the LTT plot analytically, see Figur \ref{FigLTTexpectedh}.

%The plot is `horizontal', i.e. we fix the $k$-th speciation event, and have a distribution for the time. We plot the expected time. Between events, we interpolate with a straight line.
%Note that it is also interesting to consider the ``vertical'' expected LTT plot: fix a time and calculate the expected number of species at that time. Analytic results for the vertical plots are established in \citet{Gernhard2007LTT}.

\section{Properties of the speciation times} \label{SecProperties}
\subsection{More on the point process of cBDP}
In Section \ref{SecPointProc}, we showed that a reconstructed tree of a cBDP of age $t$ can be interpreted as a point process on $n-1$ points which are i.i.d. We will see in this section, that the same is not true if we do not condition on the age of the tree but assume a uniform prior. %However, if conditioning on the $mrca$, we again obtain a point process representation with indepentent points.

From Theorem \ref{ThmFor} we obtain the density function for $x=(x_1, \ldots, x_{n-1})$, the order statistic of the speciation times, conditioned on the time of origin, $t$,
$$f(x|t,n) = (n-1)! \prod_{i=1}^{n-1} \frac{ (\lambda-\mu)^2 e^{-(\lambda-\mu)x_i}}{(\lambda-\mu e^{-(\lambda-\mu)x_i})^2} \frac{\lambda-\mu e^{-(\lambda-\mu)t}}{1 - e^{-(\lambda-\mu)t}}.$$  
With the uniform prior on the time of origin, we obtain the density for $x$ given we have $n$ extant species, $f(x|n)$,
\begin{eqnarray*}
%f(x|n) &=& \frac{d}{dx} \int_{x_1}^\infty F(x|t,n) q_{or}(t|n) dt \\
f(x|n) &=& \int_{x_1}^\infty f(x|t,n) q_{or}(t|n) dt \\
&=&    n! \lambda^{n} (\lambda-\mu)^2 
\left( \prod_{i=1}^{n-1} \frac{ (\lambda-\mu)^2 e^{-(\lambda-\mu)x_i}}{(\lambda-\mu e^{-(\lambda-\mu)x_i})^2} \right)
\int_{x_1}^\infty \frac{  e^{-(\lambda-\mu)t}}{(\lambda-\mu e^{-(\lambda-\mu)t})^{2}}dt \\
&\stackrel{\mu \neq 0 }{=}& n! \lambda^{n} (\lambda-\mu)^2 
\left( \prod_{i=1}^{n-1} \frac{ (\lambda-\mu)^2 e^{-(\lambda-\mu)x_i}}{(\lambda-\mu e^{-(\lambda-\mu)x_i})^2} \right)
\left[\frac{1}{- \mu (\lambda-\mu)  (\lambda-\mu e^{-(\lambda-\mu)t})  }   \right]_{x_1}^\infty \\
&=& n! \lambda^{n-1} (\lambda-\mu)  \frac{ e^{-(\lambda-\mu)x_1}}{ \lambda-\mu e^{-(\lambda-\mu)x_1}} 
\prod_{i=1}^{n-1} \frac{   (\lambda-\mu)^2  e^{-(\lambda-\mu)x_i}}{(\lambda-\mu e^{-(\lambda-\mu)x_i})^2} .
\end{eqnarray*}
If the $n-1$ speciation points would be i.i.d. with density function $g$, we would have $f(x|n)=(n-1)! \prod_{i=1}^{n-1} g(x_i|n)$.
Such a function $g$ does not exist due to the $x_1$, 
%Since each permutation of the $n-1$ speciation points is equally likely, we have
%$f(s|n) = n \lambda^{n-1} (\lambda-\mu)  \frac{ e^{-(\lambda-\mu)x_1}}{ \lambda-\mu e^{-(\lambda-\mu)x_1}} 
%\prod_{i=1}^{n-1} \frac{   (\lambda-\mu)^2  e^{-(\lambda-\mu)x_i}}{(\lambda-\mu e^{-(\lambda-\mu)x_i})^2} $
i.e. the $s_i$ are not i.i.d. However, since each permutation of the $s_i$ is equally likely (Corollary \ref{CorPermutation}), the $s_i$ are distributed identical.

%Since $f(s|n) \neq \prod_{i=1}^{n-1} g(s_i)$, the $n-1$ speciation points are not i.i.d.
If we condition on the time of the $mrca$, $x_1$, we again have independent points, as stated in Theorem \ref{ThmFor} (without using Theorem \ref{ThmFor}, we could also 
show the independence by calculating  the density of $x$ conditioning on $n, x_1$ as $f(x|n,x_1) = \frac{f(x|n)}{f(x_1|n)} =\frac{f(x|n)}{f_{\cA_n^1}(x_1)}$).
%it could also be established with the following procedure:
% \begin{eqnarray*}
% f(x|n,x_1) &=& \frac{f(x|n)}{f(x_1|n)} \\
% &=& \frac{f(x|n)}{{f_{\cA_n^1}(x_1)}} \\
% &=&\frac {n! \lambda^{n-1} (\lambda-\mu)  \frac{ e^{-(\lambda-\mu)x_1}}{ \lambda-\mu e^{-(\lambda-\mu)x_1}} 
% \prod_{i=1}^{n-1} \frac{   (\lambda-\mu)^2  e^{-(\lambda-\mu)x_i}}{(\lambda-\mu e^{-(\lambda-\mu)x_i})^2} }
%  {n(n-1) \lambda^{n-1} (\lambda-\mu)^{3} e^{-(\lambda-\mu)2x_1}  \frac{(1- e^{-(\lambda-\mu)x_1})^{n-2}}{(\lambda-\mu  e^{-(\lambda-\mu)x_1})^{n+1}}} \\
%  &=&\frac {(n-2)! (\lambda-\mu  e^{-(\lambda-\mu)x_1})^{n-2}
% \prod_{i=2}^{n-1} \frac{   (\lambda-\mu)^2  e^{-(\lambda-\mu)x_i}}{(\lambda-\mu e^{-(\lambda-\mu)x_i})^2} }
%  {    (1- e^{-(\lambda-\mu)x_1})^{n-2}} \\
% &=&(n-2)!
% \prod_{i=2}^{n-1} \frac{(\lambda-\mu)^2  e^{-(\lambda-\mu)x_i}}
% {(\lambda-\mu e^{-(\lambda-\mu)x_i})^2}
%  \frac{\lambda-\mu  e^{-(\lambda-\mu)x_1}}
%  {1- e^{-(\lambda-\mu)x_1}}. \\
% \end{eqnarray*}
% Since each permutation is equally likely, we have 
% $$f(s_2, \ldots, s_{n-1}|n,x_1) = 
% \prod_{i=2}^{n-1} \frac{(\lambda-\mu)^2  e^{-(\lambda-\mu)s_i}}
% {(\lambda-\mu e^{-(\lambda-\mu)s_i})^2}
%  \frac{\lambda-\mu  e^{-(\lambda-\mu)x_1}}
%  {1- e^{-(\lambda-\mu)x_1}} $$
% Therefore, $s_2, \ldots , s_{n-1}$ are i.i.d. for $\mu \neq 0$.
For $\mu = 0$, i.e. for the Yule model, we can establish the same result,
\begin{eqnarray*}
f(x|n) &=& \int_{x_1}^\infty f_{or}(x|t,n) q_{or}(t|n) dt =  n! \lambda^{n} 
\prod_{i=1}^{n-1}   e^{-\lambda x_i}
\int_{x_1}^\infty   e^{- \lambda t}dt \\
&=& n! \lambda^{n-1}  e^{- \lambda x_1} 
\prod_{i=1}^{n-1}   e^{-\lambda x_i},
 \end{eqnarray*}
i.e. the $s_i$ are not independent.
Conditioning on $x_1$, we again have independent points, as stated in Theorem \ref{ThmFor}.
% \begin{eqnarray*}
% f(x|n,x_1) &=& \frac{f(x|n)}{f(x_1|n)} \\
% &=& \frac{f(x|n)}{{f_{\cA_n^1}(x_1)}} \\
% &=& \frac{ n! \lambda^{n-1}  e^{- \lambda x_1} 
% \prod_{i=1}^{n-1}   e^{-\lambda x_i}}
%  {n(n-1)  \lambda e^{-\lambda 2 x_1}  (1- e^{-\lambda x_1})^{n-2}} \\
%  &=&  (n-2)!   
% \prod_{i=2}^{n-1}  \frac{\lambda e^{-\lambda x_i}}{1- e^{-\lambda x_1}} \\
% \end{eqnarray*}
% Since each permutation of the $s_i$ is equally likely, we have
% $$f(s|n,x_1) =\prod_{i=2}^{n-1}  \frac{\lambda e^{-\lambda s_i}}{1- e^{-\lambda x_1}} $$
% i.e. the $s_i$ are i.i.d.

\begin{rem}
Let us now consider the joint probability of the $n-1$ speciation events and the time of origin, $f(x,t|n)$. With $t:=x_0$, we have,
\begin{eqnarray*}
f(x_0,x_1,\ldots x_{n-1}|n) &=& f(x_1,\ldots x_{n-1}|t,n) q_{or}(t|n) =   n! \prod_{i=0}^{n-1} \lambda \frac{ (\lambda-\mu)^2 e^{-(\lambda-\mu)x_i}}{(\lambda-\mu e^{-(\lambda-\mu)x_i})^2},
\end{eqnarray*}
i.e. $x_0,x_1,\ldots x_{n-1}$ is the order statistic of $n$ i.i.d. random variables.
\end{rem}

\subsection{The point process of the coalescent} 
The coalescent is the standard neutral model for population genetics. The $n$ individuals in a population are assumed to coalesce as follows. For the most recent coalescent event, pick two of the $n$ individuals uniformly at random, the time between today and their coalescent is distributed exponential (rate ${n \choose 2} \lambda$) where $\lambda$ is the rate of coalescent.
We will show that this process -- even though it is very similar to the Yule process -- does not have a point process representation with i.i.d. coalescent points.

Let $x=(x_1,x_2,\ldots,x_{n-1})$ be the order statistic of the coalescent times (with $x_1>,x_2 \ldots > x_{n-1}$). Note that $x_i-x_{i+1}$ is distributed exponential with rate ${i+1 \choose 2} \lambda$.
The density function for $x$ is therefore,
\begin{eqnarray*}
f(x|n) &=& 
 \left( \lambda {n \choose 2} e^{-\lambda {n \choose 2} x_{n-1}} \right)   \prod_{i=1}^{n-2} \lambda {i+1 \choose 2} e^{- \lambda {i+1 \choose 2} (x_i - x_{i+1})} \\
&=& \frac{n! (n-1)! }{2^{n-1}} \prod_{i=1}^{n-1} \lambda e^{- \lambda i x_i}.
\end{eqnarray*}
Conditioning on the time of the most recent common ancestor, $x_1$, we get,
\begin{eqnarray*}
f(x|n,x_1) &=& \frac{f(x|n)}{f(x_1|n)} = \frac{n! (n-1)!}{f(x_1|n) 2^{n-1}} \prod_{i=1}^{n-1} \lambda e^{- \lambda i x_i} 
= h(x_1,n) \prod_{i=2}^{n-1} \lambda e^{- \lambda i x_i}.
\end{eqnarray*}
% f(x_1|n) is given in \books\coalescent3.pdf , page 19
where $h$ is a function only depending on $x_1,n$.
If the $n-2$ coalescent points would be i.i.d. with density function $g$,
we would have $f(x|n,x_1) = (n-2)! \prod_{i=2}^{n} g(x_i,x_1,n)$.
However, due to the $i$ in $e^{- \lambda i x_i}$, this property is not satisfied, therefore the $n-2$ points are not i.i.d. However, in the coalescent, also each ranked oriented tree shape is equally likely (see \citet{Aldous2001} or argue as in Theorem \ref{ThmTreeShape}), therefore each permutation of the $s_i$ has the same probability. That means that the $s_i$ are identical distributed-- but not independent.

\subsection{Backwards process of a cBDP}
%A pair of species coalesces under the Yule model (without conditioning on the number $n$ of extant species) with density $\lambda e^{-\lambda s}$.

%\subsubsection{Known tree age}
%Consider the conditioned birth-death process, where we obtain $n$ species today, after time $t$.
In the birth-death process, extant species speciate and die with exponential waiting times.
However, we condition the process to obtain a reconstructed tree with $n$ extant species today. We will describe the backward process, i.e. determine the waiting time until the extant species coalesce in the reconstructed tree.

\begin{thm}
Under the conditioned birth death process, a pair of species coalesces according to density function $f(s|t)$ from Theorem \ref{ThmFor}. 
\end{thm}

\begin{proof}
Consider a fixed pair of species out of the $n$ species. Obviously, we can put them next to each other on the $x$-axis of the point process, at location $(i,i+1)$. Their coalescent point is $(i+1/2, s_i)$, see Theorem \ref{ThmIndeps}. The time $s_i$ has the distribution with density function $f(s|t)$ from Theorem \ref{ThmFor}.
\end{proof}

In a reconstructed tree with $n$ species, the time to the last speciation event, i.e. the time between the $(n-1)$-st speciation event and today is $\cA_{n,t}^{n-1}$, $\cA_n^{n-1}$.
The time between the $k$-th and the $(k+1)$-st speciation event can be calculated as follows. 
First note that since the $n-1$ points in the point process are i.i.d. with density function $f(s|t)$,
the density function $g$ of point $j_1$ being at time $s_{j_1}$ and $j_2$ being at time $s_{j_2}$ is,
$$g(s_{j_1},s_{j_2}|t) = f(s_{j_1}|t)f(s_{j_2}|t).$$
Assume the $k$-th speciation event is at time $\tau$ and the $(k+1)$-st speciation event is at time $\tau-s$. We have $n-1$ possibilities choosing the point for the $k$-th speciation event from the $n-1$ points, and $n-2$ possibilities to choose the point for the $(k+1)$-st speciation event.
The density function for having a speciation event at time $\tau$ and $\tau-s$ is therefore $(n-1)(n-2)f(\tau|t)f(\tau-s|t)$.
The probability for $k-1$ speciation points of the remaining $n-3$ speciation points being earlier than $\tau$ is ${n-3 \choose k-1} (1-F(\tau))^{k-1}$. The probability that the remaining $n-k-2$ speciation points happend after $\tau-s$ is $F(\tau-s)^{n-k-2}$. Overall,
$$f_{\cA_{n,t}^k - \cA_{n,t}^{k+1}} (s) = \int_s^t (n-1)(n-2) {n-3 \choose k-1} (1-F(\tau|t))^{k-1} F(\tau-s|t)^{n-k-2} f(\tau|t)f(\tau-s|t) d \tau.$$

%We want to obtain the probability that the $k+1$-st speciation event was at least a time $s$ later. Assume the $k$-th speciation event is at time $\tau$. Each of the $n-1$ speciation points might be the $k$-th speciation event. We need to determine the probability having $k-1$ of the $n-2$ remaining points earlier than $\tau$, this is ${n-2 \choose k-1} (1-F(\tau|t))^{k-1}$. The probability of the remaining $n-k-1$ points being later than $\tau-s$ is $F(\tau-s|t)^{n-k-1}$.
%Overall,
%$$F(\cA_{n,t}^k - \cA_{n,t}^{k+1} \geq s) = \int_s^t (n-1) {n-2 \choose k-1} (1-F(\tau|t))^{k-1} F(\tau-s|t)^{n-k-1} f(\tau|t) d \tau.$$
%Density
%$$f_{\cA_{n,t}^k - \cA_{n,t}^{k+1}} (s) = \int_s^t (n-1)(n-2) {n-3 \choose k-1} (1-F(\tau|t))^{k-1} F(\tau-s|t)^{n-k-2} f(\tau|t)f(\tau-s|t) d \tau.$$
%

The time between the $k$-th and $l$-th speciation event ($k<l$) in a tree of age $t$ can be obtained in the same way. In addition to above, we require $l-k-1$ points to be between $\tau$ and $\tau-s$,
\begin{eqnarray*}
f_{\cA_{n,t}^k - \cA_{n,t}^{l}} (s) &=& \int_s^t (n-1)(n-2) {n-3 \choose k-1}{n-k-2 \choose l-k-1} (1-F(\tau|t))^{k-1} \times \\
& & ( F(\tau|t)-F(\tau-s|t))^{l-k-1} F(\tau-s|t)^{n-l-1} f(\tau|t)f(\tau-s|t) d \tau.
\end{eqnarray*}

%$$F(\cA_{n,t}^k - \cA_{n,t}^l \geq s) = \int_s^t \sum_{j=n-l-1}^{n-k-1} {n-1 \choose 1} {n-2 \choose k-1}{n-k-1 \choose j} f(\tau|t) (1-F(\tau|t))^{k-1} F(\tau-s|t)^{j} d \tau.$$
%If the tree age is unknown, we integrate over the time of origin, 
%$$F(\cA_{n}^k - \cA_{n}^l \geq s) = \int_0^\infty \int_s^t \sum_{j=n-l-1}^{n-k-1} {n-1 \choose k}{n-k-1 \choose j} (1-F(\tau|t))^k F(\tau-s)^{j} q(t|n) d \tau dt.$$

Note that $\cA_{n,t}^{k-1} - \cA_{n,t}^{k}$ is the time until a coalescent event for $k$ species in the reconstructed tree.
We found analytic solutions for the above integrals under the Yule model (see next Section). For $\mu \neq 0$, the densities can be derived with numerical integration in the PhyloTree package.
If assuming a uniform prior for the time of origin, we additionally need to integrate the above densities over $t$, weighted by $q_{or}(t|n)$.

\subsection{Backwards process of the Yule model}
\subsubsection{Known tree age}
For the Yule model, we can calculate the time between any two speciation events in the reconstructed tree analytically.
The time between the $k$-th speciation event and the $l$-th speciation event, $l > k$, given the time between the $n$-th speciation event (today) and the origin of the tree is $t$
%, $\cA_{n,t}^{k,l}$,
 has been calculated in \citet{Gernhard2007Yule} for $\lambda=1$ which yields for general $\lambda$ to %since f(s|t) sum of exponentials, e^s. F_{\lambda}(s|t) = F_1(\lambda s|\lambda t). So f_{\lambda}(s|t) = f_1(\lambda s|\lambda t) \times \lambda
$$f_{\cA_{n,t}^{k}-\cA_{n,t}^{l}} (s)  = \lambda  \sum_{i=0}^{k-1} \sum_{j=0}^{n-l-1} B_{i,j} e^{\lambda (n-l)s}  \frac{(e^{\lambda s}-1)^{l-k-1}}{ (e^{\lambda t}-1)^{n-1}}
  (e^{\lambda (n-k+i)( t-s)}-  e^{\lambda j( t-s)} )  
$$
with $B_{i,j}=k(k+1) {l \choose k+1}  {n-1 \choose l} {k-1 \choose i} {n-l-1 \choose j} \frac{(-1)^{n+k-l-i-j}}{n-k+i-j}$.
Note that $\cA_{n,t}^{k-1}-\cA_{n,t}^{k}$ is the time until a coalescent event of $k$ extant species.

Analogous results are not straightforward to obtain in a process with extinction. However, the expectation can be calculated straightforward for the cBDP, $ \bE[\cA_{n,t}^k -\cA_{n,t}^l ] =  \bE[\cA_{n,t}^k] -\bE[\cA_{n,t}^l ]$.

\subsubsection{Unknown tree age} \label{SecSuccSpec}

%We start with the unknown tree age, since we will need these results for the known tree age. 
We assume an uniform prior for the time of origin of a tree-- a first ancestor species was created at any point in the past with equal probability. Since we want to obtain $n$ species today, the time of origin has to be conditioned to see $n$ species today.

For the pure birth process, this is equivalent to start growing a tree and wait until the tree has $n$ species. After the $(k-1)$-th speciation event, we always have an exponential (rate $ k$) waiting time. Therefore, also the coalescent has an exponential (rate $k$) waiting time. %So under the Yule model, we have $\cA_n^{k-1}-\cA_n^k = \cA_k^{k-1}$. 
This is not true in a process with extinction.

It remains to consider the time between the $(n-1)$-st speciation event and today. Note that today is not defined as the $n$-th speciation event, but whenever we look at the process.
The density for the time between the $(n-1)$-st speciation event and today is
$$f_{\cA_n^{n-1}}(s) = n \lambda   e^{- \lambda ns}$$
which is the exponential (rate $n$) distribution. Therefore, looking at the tree today is equivalent to looking at the tree at the time of the $n$-th speciation event. This observation has been discussed in \cite{HartmannEtAl2007}.

In \citet{Gernhard2006Rank}, the time between the $k$-th and the $l$-th speciation event, $l > k$, is established for $\lambda=1$ which yields for general $\lambda$ to
$$f_{\cA_{n}^{k} - \cA_{n}^{l} }(s) = \lambda (k+1) {l \choose k+1} e^{-l \lambda s} (e^{\lambda s}-1)^{l-k-1}. $$

For the general birth-death process, the waiting times cannot be calculated straightforward, since the time to the $n$-th speciation event differs from the time until today (note that we might have the $n$-th speciation event before today -- i.e. the $n$-th speciation event is followed by extinction).
However, obtaining the expectation for $\cA_{n}^{k,l}$ is straightforward,
$ \bE[\cA_{n}^{l} - \cA_{n}^{k}] = \bE[\cA_{n}^{l}]-\bE[\cA_{n}^{k}].$

% \subsection{Edge lengths in reconstructed trees -- maybe leave out!?}
% In the following we want to derive the length of an edge $e=(u,v)$ in a reconstructed tree of age $t$, $\cA_{e,t}$.
% We will need to derive the time of vertex $u$, $\cA_{u,t}$. From \citet{Gernhard2007}, we have
% $$f_{\cA_{u,t}} = \ldots$$

% The number of leaves below $v$ is $m$. The vertex $v$ can be considered as the first speciation event in a tree which has its origin at vertex $u$ and $m$ species today. Let the age of vertex $u$ be $\tau$. Then the probability of edge $e$ having length $s$ is $f_{\cA_{m,t}^1} (\tau-s)$.
% The probability of vertex $u$ being at time $\tau$ is $f_{\cA_{u,t}}(\tau)$.  Therefore the edge length density is
% $$f_{\cA_{e,t}} (s) = \int_s^t f_{\cA_{m,t}^1} (\tau-s) f_{\cA_{u,t}}(\tau) d \tau.$$
% Assuming an uniform prior for the age of the tree, we obtain,
% $$f_{\cA_{e}} (s) = \int_0^\infty \int_s^t f_{\cA_{m,t}^1} (\tau-s) f_{\cA_{u,t}}(\tau) q(t|n) d \tau dt.$$

\section{Applications}
Knowing the density and expectation of the $k$-th speciation time given we have $n$ species today, we can obtain the density and expectation for the time of each interior node of a given tree. This can be used for dating phylogenies, if only the shape is inferred-- missing dates in phylogenies could be due to supertree methods, morphological data or absence of a molecular clock. In earlier work \citep{Gernhard2006Rank,Gernhard2007}, we gave the method and computer programs for dating phylogenetic trees. So far though, we only knew the speciation times for the Yule model and the cCBP. With the results in this paper, we can date a phylogeny assuming any constant rate birth-death model. The methods are implemented in our PhyloTree package for python.

The point process representation is useful for simulating reconstructed trees on $n$ taxa. If we have no extinction, we can simulate until obtaining $n$ species. More precise, we stop at the $n+1$-st speciation event, since that is the same as stopping today (Section \ref{SecSuccSpec}). However, with extinction, simulations are tricky -- we could return to $n$ species again and again. With the point process, it is easy to sample trees on $n$ species. First sample a time of origin according to the density $q_{or}(t|n)$ in Equation (\ref{EqnPtorn}). Then sample $n-1$ speciation times according to the density $f(s|t)$ in Theorem \ref{ThmFor}. 

If we want to simulate reconstructed trees on $n$ taxa and age $t$, direct simulation of the process seems almost impossible -- we simulate until $t$ and only keep the realization if we see $n$ species. We will throw away a lot of realizations (always if we do not see $n$ species), therefore the time amount until we have a reasonable size of samples is huge. With the point process, we simply sample $n-1$ speciation times according to the density $f(s|t)$ in Theorem \ref{ThmFor}. These sampling methods and more general sampling methods will be discussed in detail in \citet{HartmannEtAl2007}.

\section{Results and Outlook}
The $n-1$ speciation points in the point process representation are i.i.d. if conditioning on the time of origin or the most recent common ancestor. As discussed, this allows us to calculate the speciation times in a reconstructed phylogeny. So far, we calculate the speciation time with only conditioning on the shape of the phylogeny. With the point process, one might be able to condition on the shape as well as on some known dates in the phylogeny. This would be valuable for dating supertrees, since some speciation times are usually known.

In the application section, we showed that simulations of reconstructed trees become easy using the point process. This becomes useful for comparing the model with the data on aspects where no analytical results are known.

\section{Acknowledgements}
The author thanks Mike Steel, Arne Mooers, Daniel Ford, Dirk Metzler and Anusch Taraz for very helpful comments and discussions.
Financial support by the Deutsche Forschungsgemeinschaft through the graduate program ``Angewandte Algorithmische Mathematik" at the Munich University of Technology and by the Allan Wilson Center through a summer studentship is gratefully acknowledged.

\bibliographystyle{apalike}
\bibliography{bibliography1}

\appendix

\section{Proofs}

\bigskip
\noindent
{\it Proof of Theorem \ref{ThmExpt}.}
Under the Yule model, i.e. $\mu=0$, the expectation of $\cA_{n,t}^k$ has been calculated in \cite{Gernhard2007Yule} for $\lambda = 1$,
\begin{eqnarray*}
\bE[\cA_{n,t}^k(\lambda=1)] &=& \sum_{i=0}^{n-k-1} \sum_{j=0}^{k-1} \frac {k {n-1 \choose k}{n-k-1 \choose i}{k-1 \choose j}(-1)^{i+j}}{ (k+i-j)^{2}} \times \\
& & (1-e^{- t})^{1-n} (e^{-j  t} -  ( (k+i-j) t + 1)e^{-(k+i) t}).
\end{eqnarray*}
For general $\lambda$,
since
$$ f_{\cA_{n,t}^k}(s) = k {n-1 \choose k} \lambda (e^{-\lambda s} - e^{-\lambda t})^{k-1} e^{-\lambda s} \frac{ (1-e^{-\lambda s})^{n-k-1} }{ (1- e^{-\lambda t})^{n-1}}, $$
we have
$$\bE[\cA_{n,t}^k (\lambda)] = \int_0^t k {n-1 \choose k} \lambda s (e^{-\lambda s} - e^{-\lambda t})^{k-1} e^{-\lambda s} \frac{ (1-e^{-\lambda s})^{n-k-1} }{ (1- e^{-\lambda t})^{n-1}} ds.$$
Substituting $x = \lambda s$ yields
\begin{eqnarray*}
\bE[\cA_{n,t}^k (\lambda)] &=& \int_0^{\lambda t}\frac{ k}{\lambda} {n-1 \choose k} x (e^{-x} - e^{-\lambda t})^{k-1} e^{-x} \frac{ (1-e^{-x})^{n-k-1} }{ (1- e^{-\lambda t})^{n-1}} dx \\
&=& \frac{\bE[\cA_{n,\lambda t}^k (\lambda = 1)]}{\lambda}.
\end{eqnarray*}
For $0<\mu<\lambda$, we have,
\begin{eqnarray}
\bE[\cA_{n,t}^k] &=& \int_0^t s f_{\cA_{n,t}^k} (s) ds = [s F_{\cA_{n,t}^k}(s) ]_0^t  - \int_0^t  F_{\cA_{n,t}^k}(s) ds \notag \\
&=& t - \sum_{i=0}^{k-1} \sum_{j=0}^i {n-1 \choose i} {i \choose j} (-1)^{i+j} \int_0^t F(s|t)^{n-j-1} ds  \label{EqnExp} \\
&=& t - \sum_{i=0}^{k-1} \sum_{j=0}^i {n-1 \choose i} {i \choose j} (-1)^{i+j}  \left( \frac{\lambda-\mu e^{-(\lambda-\mu)t}}{1 - e^{-(\lambda-\mu)t}} \right)^{n-j-1}  \notag  \\
& & \int_0^t \left( \frac{1-e^{-(\lambda-\mu)s}}{\lambda-\mu e^{-(\lambda-\mu)s}} \right)^{n-j-1} ds  \notag \\
&=& t - \sum_{i=0}^{k-1} \sum_{j=0}^i \sum_{l=0}^{n-j-1} {n-1 \choose i} {i \choose j} {n-j-1 \choose l} (-1)^{i+j+l}  \notag \\
& &  \left( \frac{\lambda-\mu e^{-(\lambda-\mu)t}}{1 - e^{-(\lambda-\mu)t}} \right)^{n-j-1} 
 \int_0^t  \frac{e^{-(\lambda-\mu) l s} }{\left( \lambda-\mu e^{-(\lambda-\mu)s}\right)^{n-j-1}}  ds.  \notag 
\end{eqnarray}
With the substitution $x = \lambda - \mu e^{-(\lambda-\mu)s}$, we obtain for $l >0$,
\begin{eqnarray*}
\int_0^t  \frac{e^{-(\lambda-\mu) l s}}{\left( \lambda-\mu e^{-(\lambda-\mu)s} \right)^{n-j-1} } ds &=& \frac{1}{\mu(\lambda-\mu)} \int_{\lambda-\mu}^{ \lambda - \mu e^{-(\lambda-\mu)t}}  \frac{ \left( \frac{\lambda-x}{\mu}  \right)^{l-1}  }{x^{n-j-1}} dx \\
&=& \frac{1}{(\lambda-\mu) \mu^l} \sum_{m=0}^{l-1} {l-1 \choose m} (-1)^m \lambda^{l-1-m} \int_{\lambda-\mu}^{ \lambda - \mu e^{-(\lambda-\mu)t}}  x^{m+j+1-n} dx \\
&=& \frac{1}{(\lambda-\mu) \mu^l} \sum_{m=0}^{l-1} {l-1 \choose m} (-1)^m \lambda^{l-1-m} h(j,m) 
\end{eqnarray*}
where
\begin{eqnarray*}
h(j,m) &=&  
\left\{
\begin{array}{ll}
     \ln \frac{ \lambda - \mu e^{-(\lambda-\mu)t} }{\lambda-\mu},   & \hbox{if $m+j+1-n = -1$,} \\
     \frac{ (\lambda - \mu e^{-(\lambda-\mu)t})^{m+j+2-n} -  (\lambda - \mu)^{m+j+2-n} }{m+j+2-n}  & \hbox{else.} \\
\end{array}
\right.
\end{eqnarray*}
For $l=0$, we have with the substitution $x=\lambda e^{(\lambda-\mu)s} - \mu$,
\begin{eqnarray*}
g(j)&:=& \int_0^t  \frac{1}{\left( \lambda-\mu e^{-(\lambda-\mu)s} \right)^{n-j-1} } ds \\
&=& \int_0^t \frac{e^{(\lambda-\mu)(n-j-1)s}}{(\lambda e^{(\lambda-\mu)s} -\mu)^{n-j-1}}  \\
&=& \frac{1}{(\lambda-\mu)\lambda} \int_{\lambda-\mu}^{\lambda e^{(\lambda-\mu)t} - \mu}
\frac{ \left( \frac{x+\mu}{\lambda} \right)^{n-j-2}}{x^{n-j-1}} dx \\
&=& \frac{1}{(\lambda-\mu)\lambda^{n-j-1}} \sum_{m=0}^{n-j-2} {n-j-2 \choose m} \mu^m  \int_{\lambda-\mu}^{\lambda e^{(\lambda-\mu)t} - \mu}
x^{-m-1} dx \\
&=& \frac{1}{(\lambda-\mu)\lambda^{n-j-1}} \times \\
& &\left[ \ln \left( \frac{\lambda e^{(\lambda-\mu)t} - \mu}{\lambda-\mu}  \right) - \sum_{m=1}^{n-j-2} {n-j-2 \choose m} \frac{\mu^m}{m} \left( \lambda e^{(\lambda-\mu)t} - \mu)^{-m} - (\lambda-\mu)^{-m}  \right) \right].
\end{eqnarray*}
So overall,
\begin{eqnarray*}
\bE[\cA_{n,t}^k]&=& t -  \sum_{i=0}^{k-1} \sum_{j=0}^i {n-1 \choose i} {i \choose j} (-1)^{i+j} \left( \frac{\lambda-\mu e^{-(\lambda-\mu)t}}{1 - e^{-(\lambda-\mu)t}} \right)^{n-j-1} \times \\
& & \left[ g(j)+ \sum_{l=1}^{n-j-1} \sum_{m=0}^{l-1}  {n-j-1 \choose l}  {l-1 \choose m} (-1)^{l+m} \frac{\lambda^{l-1-m}}{(\lambda-\mu) \mu^l} h(j,m) \right]. 
\end{eqnarray*}
For $\mu=\lambda$, we obtain from Equation (\ref{EqnExp}),
\begin{eqnarray*}
\bE[\cA_{n,t}^k] &=& t - \sum_{i=0}^{k-1} \sum_{j=0}^i {n-1 \choose i} {i \choose j} (-1)^{i+j} \left( \frac{1+\lambda t}{t}\right)^{n-j-1}  \int_0^t \left(\frac{s}{1+ \lambda s}\right)^{n-j-1} ds.
\end{eqnarray*}
Substituting $x=1+\lambda s$, we get,
\begin{eqnarray*}
& &  \int_0^t \left(\frac{s}{1+ \lambda s}\right)^{n-j-1} ds \\
&=& \frac{1}{\lambda^{n-j}} \int_1^{1+\lambda t} \left(\frac{x-1}{x}\right)^{n-j-1} dx \\
&=& \sum_{l=0}^{n-j-1} {n-j-1 \choose l} \frac{ (-1)^l}{\lambda^{n-j}} \int_1^{1+\lambda t} x^{-l} dx \\
&=&   \frac{1}{\lambda^{n-j}} \left[  \lambda t  - (n-j-1) \ln(1+\lambda t) + \sum_{l=2}^{n-j-1} {n-j-1 \choose l} (-1)^l \frac{(1+\lambda t)^{-l+1}-1}{1-l} \right].
\end{eqnarray*}
Overall, this is
\begin{eqnarray*}
\bE[\cA_{n,t}^k] &=&  t - \sum_{i=0}^{k-1} \sum_{j=0}^i {n-1 \choose i} {i \choose j} \frac{(-1)^{i+j}}{\lambda^{n-j}} \left( \frac{1+\lambda t}{t}\right)^{n-j-1} \times \\
& & \left[  \lambda t  - (n-j-1) \ln(1+\lambda t) + \sum_{l=2}^{n-j-1} {n-j-1 \choose l} (-1)^l \frac{(1+\lambda t)^{-l+1}-1}{1-l} \right].
\end{eqnarray*}
\hfill $\Box$ 

\bigskip
\noindent
{\it Proof of Theorem \ref{ThmExp}.}
For $\mu=0$ and for $\mu = \lambda$, the expectation is established with Remark \ref{RemDensSpecModels} and Equations (\ref{EqnExpYule}) and (\ref{EqnExpCBP}).
For $\mu \neq 0$ and $\mu \neq \lambda$ we have with Theorem \ref{ThmfAnk},
$$ \bE[\cA_n^k] = \int_0^\infty (k+1) {n \choose k+1 }  \lambda^{n-k} (\lambda-\mu)^{k+2} e^{-(\lambda-\mu)(k+1)s}  \frac{(1- e^{-(\lambda-\mu)s})^{n-k-1}}{(\lambda-\mu  e^{-(\lambda-\mu)s})^{n+1}} s ds. $$
Set \begin{eqnarray*}
C_1 &:=& (k+1) {n \choose k+1 }  \lambda^{n-k} (\lambda-\mu)^{k+2}, \\
f(s) &:=&  e^{-(\lambda-\mu)(k+1)s}  \frac{(1- e^{-(\lambda-\mu)s})^{n-k-1}}{(\lambda-\mu  e^{-(\lambda-\mu)s})^{n+1}}.
\end{eqnarray*}
Therefore,
$$ \bE[\cA_n^k] = C_1 \int_0^\infty f(s) s ds = C_1 [F(s) s]_0^\infty -  C_1 \int_0^\infty F(s) ds$$
where $F(s) := \int f(s) ds$.

In the following, we calculate $F(s)$.
We do the following substitution:
$$ x = \frac{e^{-(\lambda-\mu)s}}{\lambda-\mu e^{-(\lambda-\mu)s} } \qquad \frac{dx}{ds} = - \frac{\lambda(\lambda-\mu) e^{-(\lambda-\mu)s}}{(\lambda - \mu e^{-(\lambda-\mu)s})^2} \qquad e^{-(\lambda-\mu)s} = \frac{\lambda x}{1 + \mu x}
$$
This yields
\begin{eqnarray*}
F(s) &=&  -\frac{1}{\lambda(\lambda-\mu)} \int x^{n-1} \left(\frac{1-(\lambda-\mu) x}{ \lambda x}  \right)^{n-k-1} dx \\
&=& -\frac{1}{\lambda^{n-k}(\lambda-\mu)} \int x^k (1-(\lambda-\mu)x)^{n-k-1} dx \\
&=& -\frac{1}{\lambda^{n-k}(\lambda-\mu)} \sum_{i=0}^{n-k-1} {n-k-1 \choose i} (-(\lambda-\mu))^i  \int x^{k+i} dx \\
%&=& \frac{1}{\lambda^{n-k}} \sum_{i=0}^{n-k-1} {n-k-1 \choose i} \frac{(-(\lambda-\mu))^{i-1}}{k+i+1} x^{k+i+1} \\
&=& \frac{1}{\lambda^{n-k}} \sum_{i=0}^{n-k-1} {n-k-1 \choose i} \frac{(-(\lambda-\mu))^{i-1}}{k+i+1} \left( \frac{e^{-(\lambda-\mu)s}}{\lambda-\mu e^{-(\lambda-\mu)s} }\right)^{k+i+1}.
\end{eqnarray*}
We have $\lim_{s \rightarrow \infty} F(s) s = 0$ and $F(0) \cdot 0 = 0$ and therefore,
$$ \bE[\cA_n^k] =  -  C_1 \int_0^\infty F(s) ds.$$
Substitute $x=\lambda-\mu e^{-(\lambda-\mu)s}$, 
\begin{eqnarray*}
F_2(s) &:=& \int_0^\infty F(s) ds \\
 &=&  \frac{1}{\lambda^{n-k}} \sum_{i=0}^{n-k-1} {n-k-1 \choose i} \frac{(-(\lambda-\mu))^{i-1}}{k+i+1} \int_0^\infty \left( \frac{e^{-(\lambda-\mu)s}}{\lambda-\mu e^{-(\lambda-\mu)s} }\right)^{k+i+1} ds \\
&=& \frac{1}{\lambda^{n-k}} \sum_{i=0}^{n-k-1} {n-k-1 \choose i} \frac{(-(\lambda-\mu))^{i-1}}{k+i+1} \int_{\lambda-\mu}^\lambda \frac{ \left(\frac{\lambda-x}{\mu}\right)^{k+i}  }  {\mu (\lambda-\mu)x^{k+i+1}} dx \\
&=& \frac{1}{\lambda^{n-k}} \sum_{i=0}^{n-k-1} {n-k-1 \choose i} \frac{(\lambda-\mu)^{i-2}}{k+i+1} \frac{(-1)^{i-1}}{\mu^{k+i+1}} \sum_{j=0}^{k+i} {k+i \choose j} \lambda^{j} (-1)^{k+i-j}  \int_{\lambda-\mu}^ \lambda x^{-(j+1)} dx .
\end{eqnarray*}
Evaluating the integral yields
\begin{eqnarray*}
F_2(s) &=& \frac{(-1)^{k}}{\lambda^{n-k}} \sum_{i=0}^{n-k-1} {n-k-1 \choose i} \frac{(\lambda-\mu)^{i-2}}{k+i+1} \frac{1}{\mu^{k+i+1}} \times \\
& & \left[- \log(\frac{\lambda}{\lambda-\mu})  +  \sum_{j=1}^{k+i} {k+i \choose j} \lambda^j \frac{(-1)^{j}}{j}   [\lambda^{-j} - (\lambda-\mu)^{-j} \right].
%&=& \frac{(-1)^{k}}{\lambda^{n-k}} \sum_{i=0}^{n-k-1} {n-k-1 \choose i} \frac{(\lambda-\mu)^{i-2}}{k+i+1} \frac{1}{\mu^{k+i+1}} \ldots \\
%& &  \left[- \log(\lambda-\mu e^{-(\lambda-\mu)s})  +  \sum_{j=1}^{k+i} {k+i \choose j}  \frac{(-1)^j}{j}  \left(\frac{\lambda}{\lambda-\mu e^{-(\lambda-\mu)s}}\right)^{j} \right]
\end{eqnarray*}
Therefore, with $\rho := \mu / \lambda$,
\begin{eqnarray*} \bE[\cA_n^k] &=& (k+1) {n \choose k+1 } (-1)^{k} \sum_{i=0}^{n-k-1} {n-k-1 \choose i} \frac{(\lambda-\mu)^{k+i}}{k+i+1} \frac{1}{\mu^{k+i+1}} \times \\
& & \left[ \log \left( \frac{\lambda}{\lambda-\mu}  \right)  -  \sum_{j=1}^{k+i} {k+i \choose j} \frac{(-1)^{j}}{j}  \left(1 - \left(\frac{\lambda}{\lambda-\mu}\right)^{j} \right)  \right] \\
&=&  \frac{k+1}{\lambda} {n \choose k+1 } (-1)^{k} \sum_{i=0}^{n-k-1} {n-k-1 \choose i} \frac{1}{(k+i+1) \rho} \left(\frac{1}{\rho} -1 \right)^{k+i} \times \\
& & \left[ \log \left( \frac{1}{1-\rho}  \right)  -  \sum_{j=1}^{k+i} {k+i \choose j} \frac{(-1)^{j}}{j}  \left(1 - \left(\frac{1}{1-\rho}\right)^{j} \right)  \right]
\end{eqnarray*}
which establishes the theorem.
\hfill $\Box$

\end{document}